\documentclass[12pt,reqno]{amsart}
\usepackage{amssymb}
\usepackage{amsfonts}
\usepackage{amsmath}
\usepackage{amsthm}
\usepackage{bm}
\usepackage{color}
\usepackage{enumerate}
\usepackage{epsfig}
\usepackage{hyperref}
\makeatletter
\newcommand{\Arc}{\mathrm{Arc}}
\newcommand{\Aut}{\mathrm{Aut}}
\newcommand{\Cay}{\mathrm{Cay}}

\newcommand{\Sy}{\mathrm{S}}
\newtheorem{thm}{Theorem}[section]
\newtheorem{cor}[thm]{Corollary}
\newtheorem{pro}[thm]{Proposition}
\newtheorem{lem}[thm]{Lemma}

\begin{document}

\title[Generalized quaternion groups with the $m$-DCI property]
{Generalized quaternion groups with the $m$-DCI property}
\author{Jin-Hua Xie}
\address{Jin-Hua Xie, School of Mathematics and Statistics, Beijing Jiaotong University, Beijing 100044, China}
\email{jinhuaxie@bjtu.edu.cn}

\author{Yan-Quan Feng}
\address{Yan-Quan Feng, School of Mathematics and Statistics, Beijing Jiaotong University, Beijing 100044, China}
\email{yqfeng@bjtu.edu.cn}

\author{Binzhou Xia}
\address{Binzhou Xia, School of Mathematics and Statistics, The University of Melbourne, Parkville, VIC
3010, Australia}
\email{binzhoux@unimelb.edu.au}

% ----------------------------------------------------------------
\begin{abstract}
A Cayley digraph $\Cay(G,S)$ of a finite group $G$ with respect to a subset $S$ of $G$ is said to be a CI-digraph if for every Cayley digraph $\Cay(G,T)$ isomorphic to $\Cay(G,S)$, there exists an automorphism $\sigma$ of $G$ such that $S^\sigma=T$. A finite group $G$ is said to have the $m$-DCI property for some positive integer $m$ if every Cayley digraph $\Cay(G,S)$ of $G$ with $|S|=m$ is a CI-digraph, and is said to be a DCI-group if $G$ has the $m$-DCI property for all $1\leq m\leq |G|$. Let $\mathrm{Q}_{4n}$ be a generalized quaternion group of order $4n$ with an integer $n\geq 3$, and let $\mathrm{Q}_{4n}$ have the $m$-DCI property for some $1 \leq m\leq 2n-1$. It is shown in this paper that $n$ is odd, and $n$ is not divisible by $p^2$ for any prime $p\leq m-1$. Furthermore, if $n\geq 3$ is a power of a prime $p$, then $\mathrm{Q}_{4n}$ has the $m$-DCI property if and only if $p$ is odd, and either $n=p$ or $1\leq m\leq p$.


\end{abstract}
\maketitle

\section{Introduction}
%%%%%%%%%%%%%%%%%%%%%%%%%%%%%%%%%%%%%%%%%%%%%%%%%%%%%%%%%%%%%%%%%%%%
Unless otherwise indicated, digraphs and graphs considered in this paper are finite with no parallel edges or loops, and groups are finite. For a digraph $\Gamma$, denote by $V(\Gamma)$, $E(\Gamma)$, $\Arc(\Gamma)$ and $\Aut(\Gamma)$ the vertex set, edge set, arc set, and automorphism group of $\Gamma$, respectively. If for some integer $m$, the in-valency or out-valency of every vertex of $\Gamma$ equals $m$, then we say that the digraph has \emph{in-valency $m$} or \emph{out-valency $m$}, respectively. Moreover, if the in-valency and out-valency of every vertex of a digraph both equal $m$, then we say that the digraph has \emph{valency $m$} or is \emph{$m$-valent}.

Let $G$ be a group and $S$ be a subset of $G$ with $1\notin S$. A digraph with vertex set $G$ and arc set $\{(g,sg)\mid g\in G,\,s\in S\}$ is said to be a \emph{Cayley digraph} of $G$ with respect to $S$, denoted by $\Cay(G,S)$. If $S=S^{-1}$, then both $(u,v)$ and $(v,u)$ are arcs for two adjacent vertices $u$ and $v$ in $\Cay(G,S)$, and $\Cay(G,S)$ is a graph by identifying the two arcs with one edge $\{u,v\}$. Clearly, a Cayley graph $\Cay(G,S)$ as well as its identifying Cayley digraph has the same valency $|S|$. Two Cayley digraphs $\Cay(G,S)$ and $\Cay(G,T)$ are said to be \emph{Cayley isomorphic} if $S^{\sigma}=T$ for some $\sigma \in \Aut(G)$, where $\Aut(G)$ is the automorphism group of $G$. Cayley digraphs are isomorphic if they are Cayley isomorphic, but the converse is not true. A subset $S$ of $G$ with $1\notin S$ is said to be a \emph{CI-subset} if $\Cay(G,S)\cong \Cay(G,T)$, for some $T\subseteq G$ with $1\notin T$, implies that they are Cayley isomorphic. In this case, $\Cay(G,S)$ is said to be a {\em CI-digraph}, or a {\em CI-graph} when $S=S^{-1}$. A group $G$ is said to be a DCI-group or a CI-group if all Cayley digraphs or Cayley graphs of $G$ are CI-digraphs or CI-graphs, respectively.

\'Ad\'am~\cite{Ad} conjectured that every finite cyclic group is a CI-group. Although this conjecture was disproved by Elspas and Turner~\cite{Elspas}, many researchers actively studied CI-groups and DCI-groups during the last fifty years and obtained great contributions, see \cite{Alspach,Babai,DE1,DT,Godsil1} for example. For cyclic DCI-groups and CI-groups, the classifications were finally completed by Muzychuk\cite{Mu1,Mu2}: a cyclic group of order $n$ is a DCI-group if and only if $n/\gcd(2,n)$ is square-free, and is a CI-group if and only if either $n/\gcd(2,n)$ is square-free or $n\in\{8,9,18\}$. A powerful method for studying DCI-groups or CI-groups comes from Schur ring theory, which was initiated by Schur and developed by Wielandt~(see~\cite[Chapter IV]{Wi}). In particular, this method is widely used to classify the DCI-groups and CI-groups among abelian groups, especially elementary abelian groups, refer to \cite{Feng,Kov,Mu3,MS,Sp,Sp1}. So far DCI-groups and CI-groups have been restricted to some particular groups (see \cite{Kov1,C.H.Li8}), and it is very difficult to determine whether these groups are DCI-groups or CI-groups.

For a positive integer $m$, a group $G$ is said to have the \emph{$m$-DCI property} or \emph{$m$-CI property} if all $m$-valent Cayley digraphs of $G$ are CI-digraphs or all Cayley graphs of $G$ of valency $m$ are CI-graphs, respectively. Clearly, if $G$ has the $m$-DCI property then $G$ has the $m$-CI property. A group $G$ is said to be an \emph{$m$-DCI-group} or \emph{$m$-CI-group} if $G$ has the $k$-DCI property or $k$-CI property for every positive integer $k\leq m$, respectively. Evidently, a group $G$ is a DCI-group or CI-group if $G$ has the $m$-DCI property or $m$-CI property for all $m\leq |G|$, respectively.

Considerable work has been done on the $m$-DCI property or $m$-CI property of a group, with interesting results obtained to characterize $m$-DCI-groups or $m$-CI-groups. In \cite{Fang1,Fang2,Fang3}, Fang and Xu completely classified abelian $m$-DCI-groups for a positive integer $m$ at most $3$. For an integer $n$ at least $3$ and $m\in \{1,2,3\}$, the dihedral group $\mathrm{D}_{2n}$ is $m$-DCI-group if and only if $n$ is odd~(see \cite{Qu}), and the generalized quaternion group $\mathrm{Q}_{4n}$ is $m$-DCI-group if and only if $n$ is odd~(see \cite{Ma}). In \cite{C.H.Li5}, Li, Praeger and Xu classified all finite abelian groups with the $m$-DCI property for a positive integer $m$ at most $4$, and they proposed a natural problem: characterize finite groups with the $m$-DCI property. For cyclic groups, Li~\cite{C.H.Li1} characterized the cyclic group of order $n$ with the $m$-DCI property. Soon after, Li \cite{C.H.Li6} proved that all Sylow subgroups of an abelian group with the $m$-DCI property are homocyclic. For more details, we refer to \cite{C.H.Li2,C.H.Li3,C.H.Li4,C.H.Li7} for example.

Recently, Xie, Feng and Kwon \cite{XFK} studied dihedral groups with the $m$-DCI property: if a dihedral group $G$ of order $2n$ has the $m$-DCI property for some $1\leq m\leq n-1$, then $n$ is odd and not divisible by the square of any prime less than $m$; moreover, the converse of this is true for prime power $n$, but in general it is unknown whether the converse is true.
In this paper, we consider the $m$-DCI property of generalized quaternion groups. Following \cite[(2.1)]{Ahmadi}, we call
\[
\mathrm{Q}_{4n}=\langle a,b\mid a^{2n}=1,\,b^2=a^n,\,a^b=a^{-1}\rangle
\]
the {\em generalized quaternion group} of order $4n$. Note that a generalized quaternion group is also called a dicyclic group (see \cite[Definition 1.1]{Mor}). For $n=1$, $\mathrm{Q}_{4}$ is the cylcic group of order $4$, and hence $\mathrm{Q}_4$ is a DCI-group by  \cite[Theorem 7.1]{C.H.Li7}. For $n=2$, $\mathrm{Q}_8$ is the quaternion group of order $8$, and $\mathrm{Q}_8$ is a DCI-group by \cite[Theorem~1.1]{SG}. Thus, we may assume that $n\geq 3$. For a group $G$, a subset $S$ of $G\setminus \{1\}$ is a CI-subset of $G$ if and only if the complement of $S$ in $G\setminus\{1\}$ is a CI-subset of $G$. To investigate the $m$-DCI property of $\mathrm{Q}_{4n}$, it suffices to consider $m$ such that $1\leq m\leq 2n-1$. As the first main result of this paper, we give necessary conditions for the $m$-DCI property of $\mathrm{Q}_{4n}$, which generalizes the necessary conditions for the $1$-DCI property of \cite[Lemma 3.1]{Ma}.
\begin{thm}\label{mainth1}
Let $G$ be the generalized quaternion group of order $4n$ with $n\geq 3$ such that $G$ has the $m$-DCI property for some $1\leq m\leq 2n-1$. Then $n$ is odd, and $n$ is not divisible by $p^2$ for any prime $p\leq m-1$.
\end{thm}

Based on Theorem~\ref{mainth1}, we have the following corollary, which can also be obtained from known results: $n$ is odd by \cite[Theorem 1.4]{Ma} and square free by \cite[Theorem 7.1]{C.H.Li7}.

\begin{cor}\label{cor1}
If the generalized quaternion group of order $4n$ with $n\geq 3$ is a DCI-group, then $n$ is odd and square-free.
\end{cor}

It is worth remarking that we do not know whether the converses of Theorem~\ref{mainth1} and Corollary~\ref{cor1} are true in general. However, we will show that they are true when $n$ is a prime power. Note that when $n$ is a power of a prime $p$, the conclusion in Theorem~\ref{mainth1} turn out to be that $p$ is odd and either $n=p$ or $m\leq p$.

We first illustrate the case $n=p$. Let $M$ be an abelian group of odd order such that every Sylow subgroup of $M$ is elementary abelian. Define $E(M,4)=M \rtimes \langle y\rangle$ such that $|y|=4$ and $x^y=x^{-1}$ for all $x \in M$, where $|y|=4$ is the order of $y$. By~\cite{Mu4}, $E(\mathbb{Z}_n,4)$ is a DCI-group if $(n,\varphi(n))=1$, where $\varphi(n)$ is the Euler function. Since $\mathrm{Q}_{4p}\cong E(\mathbb{Z}_p,4)$, it follows that $\mathrm{Q}_{4p}$ is a DCI-group for every odd prime $p$. In particular, the converse of Corollary~\ref{cor1} is true for prime power $n$.

Next let $n=4p^\ell$ with an odd prime $p$ and an integer $\ell\geq 2$. The following theorem asserts that $\mathrm{Q}_{4n}$ does have the $m$-DCI property for all $m\leq p$. In other words, $\mathrm{Q}_{4n}$ is a $p$-DCI-group.

\begin{thm}\label{mainth2}
Let $n\geq 3$ be a power of a prime $p$, and let $G$ be a generalized quaternion group of order $4n$. Then for $1\leq m\leq 2n-1$, $G$ has the $m$-DCI property if and only if $p$ is odd and either $n=p$ or $m\leq p$.
\end{thm}

After this Introduction, we introduce some preliminary results in Section~2. Then Theorems~\ref{mainth1} and~\ref{mainth2} will be proved in Sections 3 and 4, respectively.

%%%%%%%%%%%%%%%%%%%%%%%%%%%%%%%%%%%%%%%%%%%%%%%%%%%%%%%%%%%%%%%%%%%%
\section{Preliminaries}

In this section we give some basic concepts and facts that will be used later. For a positive integer $n$ and a prime $p$, denote by $n_p$ the largest $p$-power dividing $n$ and denote $n_{p'}=n/n_p$. Denote by $\mathrm{C}_n$ the undirected cycle of length $n$ and denote by $\overrightarrow{\mathrm{C}}_n$ the directed cycle of length $n$. Denote by $\mathrm{K}_n$ the complete graph with $n$ vertices in which two arbitrary vertices are adjacent, and denote by $\overline{\mathrm{K}}_n$ the empty graph with $n$ vertices in which no two vertices are adjacent. A digraph $\overrightarrow{\mathrm{K}}_{m,n}$ is called a \emph{complete bipartite digraph} if its vertex set can be partitioned into two subsets $X$ and $Y$ such that $|X|=m$ and $|Y|=n$ and its arc set is $\{(x,y)\mid x\in X,\,y\in Y\}$.

Let $G$ be a group. The \emph{commutator} of elements $x$ and $y$ in $G$ is $[x,y]=x^{-1}y^{-1}xy$. The \emph{derived group} $G'$ of $G$ is $\langle [x,y]\mid x, y\in G\rangle$. For a subgroup $H$ of $G$, denote the normalizer and centralizer of $H$ in $G$ by $\mathrm{N}_G(H)$ and $\mathrm{C}_G(H)$, respectively.
The following result is from \cite[Chapter 2, Theorem 1.6]{Suzuki}.

\begin{pro}\label{proper-p-group}
Let $G$ be a $p$-group for some prime $p$ and let $H$ be a proper subgroup of $G$. Then $\mathrm{N}_G(H)$ properly contains $H$, that is, $\mathrm{N}_G(H)>H$.
\end{pro}

Let $p$ be a prime. A finite group $G$ is said to be \emph{$p$-abelian} if $(xy)^p=x^py^p$ for all $x$ and $y$ in $G$. A $p$-group $G$ is called a \emph{regular $p$-group} if for arbitrary two elements $x$ and $y$ in $G$, there exists $c_1,c_2,\ldots,c_r$ in the derived group $\langle x,y\rangle'$ of $\langle x,y\rangle$ such that $(xy)^p=x^py^pc_1^pc_2^p\cdots c_r^p$. The following proposition is from \cite[Proposition 3]{Mann} and \cite[Proposition 2.3]{ZSX}.

\begin{pro}\label{regular-p-group}
Let $G$ be a $p$-group for some prime $p$. If every subgroup of $G'$ can be generated by at most $(p-1)/2$ elements, then $G$ is a regular $p$-group. Moreover, a regular $p$-group $G$ is $p$-abelian if and only if $G'$ has exponent $p$.
\end{pro}

Let $\Cay(G,S)$ be a Cayley digraph of a group $G$ with respect to $S$. For a given $g\in G$, the right multiplication $R(g)$ is a permutation on $G$ such that $x^{R(g)}=xg$ for every $x\in G$. Clearly, $R(g)$ is an automorphism of $\Cay(G,S)$. Let $R(G)=\{R(g)\mid g\in G\}$. Then $R(G)$ is a regular group of automorphisms of $\Cay(G,S)$, which is called \emph{the right regular representation} of $G$. The following well-known Babai's criterion is from \cite{Babai}~(also see \cite[Theorem 2.4]{C.H.Li8}).

\begin{pro}\label{CI-graph-prop}
A Cayley digraph $\Cay(G,S)$ is a CI-digraph if and only if every regular subgroup of $\Aut(\Cay(G,S))$ isomorphic to $G$ is conjugate to $R(G)$ in $\Aut(\Cay(G,S))$.
\end{pro}

The following result says that the $m$-DCI property of a group is hereditary by subgroups, which can be proved by the same argument as that for the $m$-CI property in~\cite[Lemma~8.2]{C.H.Li7}.

\begin{pro}\label{subgroup}
Suppose that a finite group $G$ has the $m$-DCI property for a positive integer $m$. Then every subgroup of $G$ has the $m$-DCI property.
\end{pro}

Li~\cite[Theorem~1.2]{C.H.Li1} characterized cyclic groups with the $m$-DCI property. We restate this result as follows.

\begin{pro}\label{cyclicgroupMP}
Let $G$ be a cyclic group of order $n$ such that $G$ has the $m$-DCI property for some $p+1\leq m\leq n-(p+2)$ with $p$ a prime. Then either $n=p^2$ and $m\equiv 0$ or $-1\pmod{p}$, or $n_p$ divides $\mathrm{lcm}(4,p)$.
\end{pro}

For subsets of a cyclic group, we have the following result (see \cite[Lemma~2.1]{C.H.Li6}).

\begin{lem}\label{cyclic group property}
Let $G=\langle z\rangle$ be a cyclic group of order $n$, and let $i,j\in \{1,2,\ldots,n-2\}$. If $\{z,z^2,\ldots,z^i\}=\{z^j, z^{2j},\ldots.z^{ij}\}$, then $j=1$.
\end{lem}

Let $G$ be a finite group. If for any two subgroups $H$ and $K$ of $G$, every isomorphism from $H$ to $K$ can be extended to an automorphism of $G$, then $G$ is called {\em homogeneous}. For generalized quaternion groups $\mathrm{Q}_{4n}$, the following property is shown in~\cite[Lemma~2.4]{Ma}.

\begin{lem}\label{homogeneous}
For an odd positive integer $n$, the generalized quaternion group $\mathrm{Q}_{4n}$ is homogeneous.
\end{lem}

From~\cite[Lemma~3.1]{XFK}, we have the following lemma, which provides a fairly general way to construct isomorphic Cayley digraphs.

\begin{lem}\label{general}
Let $G$ be a finite group with $L\unlhd G$ and $L\leq M\leq G$. Suppose that $A$ and $B$ are subsets of $M\setminus\{1\}$ such that $A^\gamma=B$ for some $\gamma\in\Aut(M)$ and $\gamma$ fixes every coset of $L$ in $M$, and that $C\subseteq G\setminus L$ is a union of some cosets of $L$ in $G$. Then $\Cay(G,A\cup C)\cong\Cay(G,B\cup C)$.
\end{lem}

Let $\Gamma$ be a digraph and let $X\subseteq V(\Gamma)$. The \emph{induced subdigraph} $[X]$ of $\Gamma$ by $X$ is the digraph whose vertex set is $X$ and arc set is $\{(u,v)\mid u,v\in X,\,(u,v)\in \Arc(\Gamma)\}$. Let $N$ be a subgroup of $\Aut(\Gamma)$. Denote by $u^N$ the orbit of $N$ containing $u\in V(\Gamma)$, and by $\Gamma^+(u)$ the out-neighborhood of $u$ in $\Gamma$. The \emph{quotient digraph} $\Gamma_N$ of $\Gamma$ induced by $N$ is defined as the digraph whose vertex set is the set of $N$-orbits in $V(\Gamma)$ such that $(u^N,v^N)$ is an arc of $\Gamma_N$, where $u^N$ and $v^N$ are distinct orbits of $N$, if and only if $(x,y)$ is an arc of $\Gamma$ for some $x\in u^N$ and $y\in v^N$. The digraph $\Gamma$ is said to be an {\em $N$-cover} of $\Gamma_N$, if for every $u\in \Gamma$, the out-valency of $u$ in $\Gamma$ is the same as the out-valency of $u^N$ in $\Gamma_N$, is said to be \emph{$G$-locally primitive} if $G_u$ acts primitively on $\Gamma^+(u)$ for every $u\in V(\Gamma)$, and is said to be {\em strongly connected} if there exists a directed path from $u$ to $v$ for each pair of vertices $u$ and $v$. To avoid trivial cases, a digraph with one vertex is also called strongly connected. It well known that every finite connected vertex-transitive digraph is strongly connected (see~\cite[Lemma 2.6.1]{GR} for instance). For convenience, the complete graph on two vertices is also viewed as a directed cycle.

The following is a generalization of Praeger \cite[Theorem 4.1]{Praeger} to digraph.

\begin{lem}\label{no-arc}
Let $\Gamma$ be a finite connected $G$-vertex-transitive digraph, where $G\leq \Aut(\Gamma)$, and let $N$ be a normal subgroup of $G$ with at least two orbits on $V(\Gamma)$. Then the following statements hold:
\begin{enumerate}[{\rm (a)}]
\item If $\Gamma$ is $G$-arc-transitive, then there are no arcs in the induced subdigraph of any orbit of $N$ in $\Gamma$.
\item If $\Gamma$ is $G$-locally primitive, then either $N$ is the kernel of $G$ on $V(\Gamma_N)$ acting semiregularly on $V(\Gamma)$, and $\Gamma$ is an $N$-cover of $\Gamma_N$ with $|V(\Gamma_N)|\geq 3$, or $\Gamma_N$ is a directed cycle.
\end{enumerate}
\end{lem}

\begin{proof}
To prove part~(a), let $\Gamma$ be $G$-arc-transitive and suppose on the contrary that the induced subdigraph of some orbit of $N$ has an arc. Since $\Gamma$ is $G$-vertex-transitive, it follows that the induced subdigraph of every orbit of $N$ has an arc. By the connectivity of $\Gamma$, there is an arc between some distinct orbits of $N$, say $O_1$ and $O_2$. Let $(u,v)$ be an arc of $\Gamma$ with $u\in O_1$ and $v\in O_2$. Since $O_1$ has an arc and $N$ is transitive on $O_1$, there is an arc $(u,w)$ of $\Gamma$ with $w\in O_1$. Since $\Gamma$ is $G$-arc-transitive, there exists $g\in G$ such that $u^g=u$ and $v^g=w$. However, such an element $g$ does not preserve the set of $N$-orbits as $u,w\in O_1$ and $v\in O_2$. This contradicts the fact that $N$ is normal in $G$, completing the proof of part~(a).

In following we prove part~(b). Let $\Gamma$ be $G$-locally primitive. Since $\Gamma$ is connected, there exists an arc between some distinct orbits of $N$, say $O_1$ and $O_2$. Let $(u,v)$ be an arc of $\Gamma$ with $u\in O_1$ and $v\in O_2$.

First assume that the out-neighbors of $u$ are contained in the same orbit of $N$. Then $\Gamma^+(u)\subseteq O_2$ as $v\in O_2$. Since $N$ is transitive on both $O_1$ and $O_2$, we have $\Gamma^+(x)\subseteq O_2$ for all $x\in O_1$. Since $G$ has an element mapping $O_1$ to $O_2$, the out-neighborhood of each vertex in $O_2$ is a subset of some orbit of $N$. Repeating this argument, we see that the out-neighborhood of each vertex in every orbit of $N$ is a subset of some orbit of $N$. Then we conclude from the connectivity of $\Gamma$ that $\Gamma_N$ is a directed cycle.

Now assume that the out-neighbors of $u$ are not contained in the same orbit of $N$. Let $\mathcal{O}=\{O_1,O_2,\ldots,O_n\}$ be the set of orbits of $N$ and assume that the out-neighborhood of $O_1$ in $\Gamma_N$ is $\{O_2,O_3,\ldots,O_d\}$. Then $|V(\Gamma_N)|=n\geq d\geq 3$, and $|\Gamma^+(u)\cap O_i|\geq 1$ for each $i\in\{2,\dots,d\}$. The hypothesis of part~(b) implies that $\Gamma$ is strongly connected and $G$-arc-transitive, whence $G_u$ is transitive on $\Gamma^+(u)$. Moreover, the conclusion of part~(a) implies that
\[
\{\Gamma^+(u)\cap O_2,\Gamma^+(u)\cap O_3,\ldots, \Gamma^+(u)\cap O_d\}
\]
is a partition of $\Gamma^+(u)$. Since $N$ is normal in $G$, it follows that this partition is preserved by $G_u$. Then we conclude from the $G$-local-primitivity of $\Gamma$ that $|\Gamma^+(u)\cap O_i|=1$ for each $i\in\{2,\dots,d\}$. Hence $u$ has the same out-valency as $O_1$ in $\Gamma_N$, which means that $\Gamma$ is an $N$-cover of $\Gamma_N$. Let $K$ be the kernel of $G$ acting on $\mathcal{O}$. Then the stabilizer $K_u$ fixes $\Gamma^+(u)$ pointwise because $|\Gamma^+(u)\cap O_i|=1$ for each $i\in\{2,\dots,d\}$. This implies that $K_u=K_w$ for every $w\in \Gamma^+(u)$. Then the strong connectivity of $\Gamma$ implies that $K_x=1$ for all $x\in V(\Gamma)$, that is, $K$ is semiregular on $V(\Gamma)$. Noting $N\leq K$, we deduce by the Frattini argument that $K=NK_x=N$. This shows that $N$ is the kernel of $G$ acting on $V(\Gamma_N)$ and is semiregular on $V(\Gamma)$.
\end{proof}
%%%%%%%%%%%%%%%%%%%%%%%%%%%%%%%%%%%%%%%%%%%%%%%%%%%%%%%%%%%%%%%%%%%%
\section{Proof of Theorem~\ref{mainth1}}

By \cite[Lemma 3.1]{Ma}, if $\mathrm{Q}_{4n}$ ($n\geq 3$) has the $1$-DCI property, then $n$ is odd. This is true for every $1\leq m\leq 2n-1$ as the following lemma states.

\begin{lem}\label{n-odd}
Let $G$ be a generalized quaternion group of order $4n$ with $n\geq 3$ such that $G$ has the $m$-DCI property for some $1\leq m\leq 2n-1$. Then $n$ is odd.
\end{lem}

\begin{proof}
Suppose for a contradiction that $n$ is even. Then $n\geq 4$ as $n\geq 3$. Let $G=\mathrm{Q}_{4n}=\langle a,b\mid a^{2n}=1,\,b^2=a^n,\,a^b=a^{-1}\rangle$. Then $|a|=2n$, where $|a|$ is the order of $a$, and hence $|a^2|=n$. Note that $\langle a^2\rangle$ is a characteristic subgroup of $G$ of index $4$ and $b^2=a^n\in \langle a^2\rangle$. Furthermore, $G=\langle a^2\rangle \cup b\langle a^2\rangle \cup a\langle a^2\rangle \cup ba\langle a^2\rangle$. Define
\[
\varphi: x\mapsto x \text{ for } x\in \langle a^2\rangle \cup b\langle a^2\rangle \text{ and }x\mapsto bx \text{ for } x\in a\langle a^2\rangle \cup ba\langle a^2\rangle.
\]
Then $\varphi$ fixes every element in $\langle a^2\rangle \cup b\langle a^2\rangle$ and acts on $a\langle a^2\rangle \cup ba\langle a^2\rangle$ as same as the restriction of the left multiplication of $b$ on $G$. Thus, $\varphi$ is a permutation of order $4$ on $G$, and interchanges $a\langle a^2\rangle$ and $ba\langle a^2\rangle$. First we prove a claim.

\medskip
\noindent{\bf Claim:} Let $H\subseteq \langle a^2\rangle$ and $K\subseteq a\langle a^2\rangle$ such that $H^{-1}=H$, $K^{-1}=K$ and $a^nK=K$. Then $\varphi$ is an isomorphism from $\Gamma=\Cay(G,bH\cup K)$ to $\Sigma=\Cay(G,bH\cup bK)$.

Let $(u,v)$ be an arc of $\Gamma$. Then $v=su$ for some $s\in bH\cup K$. First assume that $s\in bH$. Note that $bH\subseteq b\langle a^2\rangle$. If $u\in \langle a^2\rangle \cup b\langle a^2\rangle$, then $v=su\in \langle a^2\rangle \cup b\langle a^2\rangle$. It follows that $u^\varphi=u$ and $v^\varphi=v=su$, which implies that $(u^\varphi,v^\varphi)$ is an arc of $\Sigma$ because $s\in bH$. If $u\in a\langle a^2\rangle \cup ba\langle a^2\rangle$, then $v=su\in a\langle a^2\rangle \cup ba\langle a^2\rangle$, and so $u^\varphi=bu$ and $v^\varphi=bv=bsu=bsb^{-1}(bu)$. Since $H=H^{-1}\subseteq \langle a^2\rangle$, we have $bsb^{-1}\in b(bHb^{-1})=bH^{-1}=bH$, which implies that $(u^\varphi,v^\varphi)$ is an arc of $\Sigma$. Next assume that $s\in K$. Note that $K\subseteq a\langle a^2\rangle$. If $u\in \langle a^2\rangle \cup b\langle a^2\rangle$, then $v=su\in a\langle a^2\rangle \cup ba\langle a^2\rangle$, which implies that $u^\varphi=u$ and $v^\varphi=bv=bsu$. Since $bs\in bK$, it follows that $(u^\varphi,v^\varphi)$ is an arc of $\Sigma$. If $u\in a\langle a^2\rangle \cup ba\langle a^2\rangle$, then $v=su\in \langle a^2\rangle \cup b\langle a^2\rangle$, which implies that $u^\varphi=bu$ and $v^\varphi=v=su=sb^{-1}bu=b^{-1}s^{-1}bu$ as $s\in \langle a\rangle$. Since $K^{-1}=K$ and $a^nK=K$, we have that $b^{-1}s^{-1}=ba^ns^{-1}\in ba^nK^{-1}=bK$, and so $(u^\varphi,v^\varphi)$ is an arc of $\Sigma$. Thus, in every case, $(u^\varphi,v^\varphi)$ is an arc of $\Sigma$, and hence $\varphi$ is an isomorphism from $\Gamma$ to $\Sigma$, as claimed.

\smallskip
Note that every element in $G\setminus \langle a\rangle$ has order $4$ and has the form $ba^i$ with $1\leq i\leq 2n$. Since $\langle a\rangle$ is a characteristic subgroup of $G$, we obtain that
\begin{equation}\label{Eqnau}
a^\alpha\in \langle a\rangle \text{ and } (ba^i)^\alpha\not\in\langle a\rangle, \text{ for every } \alpha\in\Aut(G) \text{ and } 1\leq i\leq 2n.
\end{equation}
By hypothesis, $G$ has the $m$-DCI property for some $1\leq m\leq 2n-1$. Since $n$ is even, we have $m\neq 1$ by \cite[Lemma~3.1]{Ma}, and thus  $2\leq m\leq 2n-1$.

Suppose $m=2$. Take $S=\{b,b^{-1}\}$ and $T=\{a^{n/2},a^{3n/2}\}$. It is not difficult to see that $\Cay(G,S)\cong n\mathrm{C}_4\cong \Cay(G,T)$, where $n\mathrm{C}_4$ is a disjoint union of $n$ $4$-cycles. Then the $2$-DCI property of $G$ implies that there is an automorphism of $G$ mapping $S$ to $T$, contradicting \eqref{Eqnau}.

Suppose $m=3$. Take $S=\{b,b^{-1},b^2\}$ and $T=\{a^{n/2},a^{3n/2},a^n\}$. Then $\Cay(G,S)\cong n\mathrm{K}_4\cong \Cay(G,T)$, where $n\mathrm{K}_4$ is a disjoint union of $n$ copies of $\mathrm{K}_4$.  The $3$-DCI property of $G$ gives an automorphism of $G$ that maps $S$ to $T$, contradicting \eqref{Eqnau}.

Suppose $m=4, 5, 6, 7$. It follows from $|a|=2n\geq 8$ that $a^2\neq a^{-2}$. Take $K=\{a,a^{-1},a^{n+1},a^{n-1}\}$, and $H=\emptyset$, $\{1\}$, $\{1,a^n\}$ or $\{1,a^2,a^{-2}\}$, respectively. Then we derive from the Claim that $\Cay(G,bH\cup bK)\cong \Cay(G,bH\cup K)$. Since $G$ has the $m$-DCI property, there exists an automorphism of $G$ mapping $bH\cup bK$ to $bH\cup K$, contradicting \eqref{Eqnau}.

Now we may assume that $8\leq m\leq 2n-1$. Write $m=8k+j$, where $0\leq j\leq 7$ and $k\geq 1$. It follows that $4k<n$. Set
\begin{align*}
H_1&=\{a^2,a^4,\ldots, a^{2k}, a^{2n-2},a^{2n-4},\ldots, a^{2n-2k}\},\\
K_1&=\{a,a^3,\ldots, a^{2k-1}, a^{2n-1},a^{2n-3},\ldots, a^{2n-(2k-1)}\}.
\end{align*}
Then $H_1^{-1}=H_1\subseteq \langle a^2\rangle$, $K_1^{-1}=K_1\subseteq a\langle a^2\rangle$,
\begin{align*}
a^nH_1&=\{ a^{n+2},a^{n+4},\ldots, a^{n+2k}, a^{n-2},a^{n-4},\ldots, a^{n-2k}\},\\
a^nK_1&=\{a^{n+1},a^{n+3},\ldots, a^{n+2k-1}, a^{n-1},a^{n-3},\ldots, a^{n-(2k-1)}\},
\end{align*}
$(a^nH_1)^{-1}=a^nH_1\subseteq \langle a^2\rangle$, and $(a^nK_1)^{-1}=a^nK_1\subseteq a\langle a^2\rangle$. Moreover, we observe from $4k<n$ that $|H_1|=|K_1|=|a^nH_1|=|a^nK_1|=2k$. Suppose $H_1\cap a^nH_1\neq \emptyset$. Let $x\in H_1\cap a^nH_1$. Note that
\[
(H_1\cap a^nH_1)^{-1}=H_1^{-1}\cap (a^nH_1)^{-1}=H_1\cap a^nH_1.
\]
Since $x\in H_1$, we may assume $x=a^{2e}$ for some $e\in \{1,\ldots,k\}$, and then we derive from $x\in a^nH_1$ that $a^{2e}=x=a^{n-2f}$ for some $f\in \{1,\ldots,k\}$. This implies that $a^{n+2(e+f)}=1$, which is impossible because $n+2(e+f)\leq n+4k<2n$. Thus, $H_1\cap a^nH_1=\emptyset$, and so $|H_1\cup a^nH_1|=4k$. Similarly, if $K_1\cap a^nK_1\neq\emptyset$, then we can obtain $a^{n+2(e+f-1)}=1$ for some $e,f\in\{1,\ldots,k\}$, which is also impossible because $n+2(e+f-1)\leq n+4k<2n$. Thus, $K_1\cap a^nK_1=\emptyset$, and so $|K_1\cup a^nK_1|=4k$.

Note that the results of the above paragraph are proved under the condition $4k<n$, which is a consequence of the assumption. If further $4k+2<n$, then we set
\[
H_2=H_1\cup\{a^{2(k+1)},a^{2n-2(k+1)}\} \text{ and } K_2=K_1\cup \{a^{2(k+1)-1},a^{2n-2k-1}\}.
\]
Then a similar argument to the above paragraph implies that $H_2^{-1}=H_2\subseteq \langle a^2\rangle$, $K_2^{-1}=K_2\subseteq a\langle a^2\rangle$, $(a^nH_2)^{-1}=a^nH_2\subseteq \langle a^2\rangle$, $(a^nK_2)^{-1}=a^nK_2\subseteq a\langle a^2\rangle$, $|H_2|=|K_2|=|a^nH_2|=|a^nK_2|=2k+2$, $|H_2\cup a^nH_2|=|K_2\cup a^nK_2|=4k+4$. Recall that $m=8k+j$ with $k\geq 1$ and $0\leq j\leq 7$. We now discuss several cases according to $j$.

For $j=0$, write $H=H_1\cup a^nH_1$ and $K=K_1\cup a^nK_1$. We deduce from the Claim that $\Cay(G,bH\cup bK)\cong \Cay(G,bH\cup K)$. Then the $m$-DCI property of $G$ provides an automorphism of $G$ mapping $bH\cup bK$ to $bH\cup K$, contradicting \eqref{Eqnau}. For $j=1$ or $j=2$, we have the same contradiction by taking $K=K_1\cup a^nK_1$ and $H=H_1\cup a^nH_1\cup\{1\}$ or $H_1\cup a^nH_1\cup\{1,a^n\}$, respectively. For $j=3$, we have $m=8k+3\leq 2n-1$ and hence $4k+2\leq n$. If $4k+2=n$, then we take $H=H_1\cup a^nH_1\cup \{1\}$ and $K=a\langle a^2\rangle$, and if $4k+2<n$, then we take $H=H_2\cup a^nH_1\cup \{1\}$ and $K=K_1\cup a^nK_1$. Similarly, the same contradiction for \eqref{Eqnau} occurs.

For $j=4,5,6,7$, we have $4k+2<n$ as $m=8k+j\leq 2n-1$. We take $K=K_2\cup a^nK_2$, and $H=H_1\cup a^nH_1$, $H_1\cup a^nH_1\cup\{1\}$,  $H_2\cup a^nH_1$ or $H_2\cup a^nH_1\cup \{1\}$, respectively. By the Claim, $\Cay(G,bH\cup bK)\cong \Cay(G,bH\cup K)$, and then the $m$-DCI property implies that there is an automorphism of $G$ mapping $bH\cup bK$ to $bH\cup K$, contradicting \eqref{Eqnau}.
\end{proof}

%%%%%%%%%%%%%%%%%%%%%%%%%%%%%%%%%%%%%%%%%%%%%%%%%%%%%%%%%%%%%%%%%%%%
For the group $G=\mathrm{Q}_{4n}$ with the $m$-DCI property and a prime divisor $p$ of $n$ such that $p+1\leq m\leq 2n-1$, we have  the following result.

\begin{lem}\label{p-odd}
Let $G$ be a generalized quaternion group of order $4n$ with $n\geq 3$. If $G$ has the $m$-DCI property such that  $p+1\leq m\leq 2n-1$ for some prime divisor $p$ of $n$, then $p$ is odd and $n$ is not divisible by $p^2$.
\end{lem}

\begin{proof}
Let $G=\mathrm{Q}_{4n}=\langle a,b\mid a^{2n}=1,\,b^2=a^n,\,a^b=a^{-1}\rangle$. Suppose that $G$ has the $m$-DCI property such that  $p+1\leq m\leq 2n-1$ for some prime divisor $p$ of $n$. By Lemma~\ref{n-odd}, $n$ is odd, and so $p$ is odd.

Write $n'=2n/p$, $z=a^{n'}$ and $P=\langle z\rangle$. Then $n'$ is even and $P$ is the unique subgroup of order $p$ in $G$, which implies that $P$ is characteristic in $G$. Suppose for a contradiction that $p$ divides $n'$. Note that $\langle a\rangle$ has the $m$-DCI property by Proposition~\ref{subgroup}. Then it follows from Proposition~\ref{cyclicgroupMP} that $2n-(p+1)\leq m\leq 2n-1$.
Define an integer $j\in\{1,\dots,p-2\}$ and a subset $Q$ of $G$ as follows:
\[
(j,Q)=\begin{cases}
(m\bmod p,\,\emptyset),&\text{ if }m\not\equiv 0\text{ or}\,-1\ (\bmod\ p)\\
(p-2,\,\{b\}),&\text{ if }m\equiv -1\ (\bmod\ p)\\
(p-2,\,\{b,bz\}),&\text{ if }m\equiv 0\ (\bmod\ p).\\
\end{cases}
\]
Then $m=kp+j+|Q|$ for some positive integer $k\leq n'-1$.
Write $X=\langle z,b\rangle=\langle z\rangle\rtimes \langle b\rangle$. Then $X$ has an automorphism $\gamma$ induced by $z\mapsto z^{-1}$ and $b\mapsto b$. Let $Z=\{z,\ldots,z^j\}$ and let
\begin{align*}
S&=aP\cup (baP\cup\cdots\cup ba^{k-1}P) \cup (Z\cup Q),\\
T&=aP\cup (baP\cup\cdots\cup ba^{k-1}P) \cup (Z^\gamma \cup Q^\gamma).
\end{align*}
Note that $|S|=|T|=m$. Taking $L=P$ and $M=X$ and $C=aP\cup (baP\cup\cdots\cup ba^{k-1}P)$ in Lemma~\ref{general}, we obtain $\Cay(G,S)\cong\Cay(G,T)$. Since $G$ has the $m$-DCI property, we have $S^\sigma=T$ for some $\sigma\in \Aut(G)$. Let $x\in aP$. Then $x=az^\ell=a^{\ell n'+1}$ for some $0\leq \ell\leq p-1$, which implies that $|x|=2n/(2n,\ell n'+1)$. Since $2$ divides $n'$ and $p$ divides $n'$, we have $(2n,\ell n'+1)=1$, and so $|x|=2n$. This means that every element in $aP$ has order $2n$. Note that every element in $(baP\cup\cdots\cup ba^{k-1}P)\cup Q\cup Q^\gamma$ has order $4$ and every element in $Z\cup Z^\gamma$ has order $p$. We derive from $S^\sigma=T$ that $(aP)^\sigma=aP$ and $Z^\sigma=Z^\gamma$. Since $\langle a\rangle$ is characteristic in $G$, it follows that $a^\sigma=a^r$ for some integer $r$. In particular,
\[
\{z^r,\ldots,z^{jr}\}=Z^\sigma=Z^\gamma=\{z^{-1},\ldots,z^{-j}\}.
\]
Then by Lemma~\ref{cyclic group property}, $r\equiv-1\pmod{p}$. Note that $P^\sigma=P$ as $P$ is characteristic in $G$. We conclude that $aP=(aP)^\sigma=a^\sigma P^\sigma=a^rP$, which leads to $a^{r-1}\in P=\langle a^{n'}\rangle$. However, this together with $p$ dividing $n'$ implies that $p$ divides $r-1$, contradicting $r\equiv-1\pmod p$. Thus $p$ does not divide $n'$, which means that $n$ is not divisible by $p^2$, completing the proof.
\end{proof}

%%%%%%%%%%%%%%%%%%%%%%%%%%%%%%%%%%%%%%%%%%%%%%%%%%%%%%%%%%%%%%%%%%%%
Now we are ready to prove Theorem~\ref{mainth1}.

\begin{proof}[Proof of Theorem~\ref{mainth1}.]
Let $G$ be a generalized quaternion group of order $4n$ with $n\geq 3$ such that $G$ has the $m$-DCI property for some $1\leq m\leq 2n-1$. Then $n$ is odd as Lemma~\ref{n-odd} asserts. Furthermore, for any prime $p\leq m-1$, according to Lemma~\ref{p-odd} we have that $n$ is not divisible by $p^2$. This completes the proof.
\end{proof}

%%%%%%%%%%%%%%%%%%%%%%%%%%%%%%%%%%%%%%%%%%%%%%%%%%%%%%%%%%%%%%%%%%%%
\section{Proof of Theorem~\ref{mainth2}}
%%%%%%%%%%%%%%%%%%%%%%%%%%%%%%%%%%%%%%%%%%%%%%%%%%%%%%%%%%%%%%%%%%%%
In this section, we prove Theorem~\ref{mainth2}. Besides being important ingredients of the proof of Theorem~\ref{mainth2}, the following two lemmas are of their own interest as well.

\begin{lem}\label{sub-conjugate}
Let $G\leq A\leq\mathrm{Sym}(\Omega)$ with $G$ regular on $\Omega$, and let $H$ be a normal subgroup of odd order $n$ in $G$. Suppose $G=H\rtimes\langle b\rangle$ for some $b\in G$ with $|b|\in\{2,4\}$ such that either $G=H\times\langle b\rangle$ or $h^b=h^{-1}$ for all $h\in H$. Then for a regular subgroup $X$ of $A$ isomorphic to $G$, the subgroups $X$ and $G$ are conjugate in $A$ if and only if $H$ and the unique subgroup of order $n$ of $X$ are conjugate in $A$.
\end{lem}

\begin{proof}
By the assumption of the lemma, there exists $r=\pm1$ such that $h^b=h^r$ for all $h\in H$. Let $X$ be a subgroup of $A$ isomorphic to $G$. Then $X$ has a unique subgroup of order $n$, say $Y$, and we may write $X=Y\rtimes\langle c\rangle$ such that $|b|=|c|$ and $y^c=y^r$ for all $y\in Y$. We need to prove that $G$ and $X$ are conjugate in $A$ if and only if $H$ and $Y$ are conjugate in $A$. The necessity is clear because $H$ and $Y$ are the unique subgroups of order $n$ in $G$ and $X$, respectively. To finish the proof, assume that $A$ has an element $\alpha$ with $Y^\alpha=H$, and we shall show that there exists an element of $A$ conjugating $X$ to $G$.

Since $c\in \mathrm{N}_A(Y)$, we have $c^\alpha\in \mathrm{N}_A(Y^\alpha)=\mathrm{N}_A(H)$. Hence both the $2$-elements $b$ and $c^\alpha$ are in $\mathrm{N}_A(H)$. Let $P$ be a Sylow $2$-subgroup of $\mathrm{N}_A(H)$ such that $b\in P$. By Sylow Theorem, there exists $\beta \in \mathrm{N}_{A}(H)$ such that $(c^\alpha)^\beta \in P$. Then
\[
X^{\alpha\beta}=(Y\rtimes\langle c\rangle)^{\alpha\beta}=Y^{\alpha\beta}\rtimes\langle c^{\alpha\beta}\rangle
=H^\beta\rtimes\langle c^{\alpha\beta}\rangle=H\rtimes \langle c^{\alpha\beta}\rangle.
\]
Let $d=c^{\alpha\beta}\in P$. Then $|d|=|c|=|b|$ and $h^d=h^r$ for all $h\in H$ as $y^c=y^r$ for all $y\in Y$.

Write $m=|b|$. Then $m=2$ or $4$. The regularity of $G$ on $\Omega$ implies $|\Omega|=|G|=|b||H|=mn$. Since $H$ is a normal subgroup of $G$, it follows that $H$ has $m$ orbits on $\Omega$, say $\Omega_1,\Omega_2,\ldots,\Omega_m$, where  $|\Omega_i|=n$ for every $i\in\{1,\dots, m\}$. Moreover, since $G=H\rtimes\langle b\rangle$ with $|b|=m$, the element $b$ permutes the set $\{\Omega_1,\Omega_2,\ldots,\Omega_m\}$ cyclicly. Similarly, $d$ permutes $\{\Omega_1,\Omega_2,\ldots,\Omega_m\}$ cyclicly because $X^{\alpha\beta}=H\rtimes\langle d\rangle$ with $|d|=m$.

Note that $P$ is a $2$-group and $b,d\in P$. Every orbit of $P$ on $\Omega$ has length $2$-power that is at least $m$, where $m=2$ or $4$. If every orbit of $P$ on $\Omega$ has length greater than $m$, then every orbit of $P$ on $\Omega$ has length divisible by $2m$, and so $|\Omega|$ is divisible by $2m$, which is impossible because $|\Omega|=mn$ with $n$ odd. Thus $P$ has an orbit of length $m$, say $\Delta$. In particular, both $\langle b\rangle$ and $\langle d\rangle$ are regular on $\Delta$. Write $\Delta=\{\delta_1,\delta_2,\ldots,\delta_m\}$. Since $b$ permutes $\{\Omega_1,\Omega_2,\ldots,\Omega_m\}$ cyclicly, we have $|\Delta\cap \Omega_i|=1$, say $\delta_i\in \Omega_i$, for every $i\in\{1,\dots,m\}$.

Consider $b^{\Delta}$ and $d^{\Delta}$, namely, the permutations of $\Delta$ induced by $b$ and $d$, respectively. For $m=2$, since both $\langle b\rangle$ and $\langle d\rangle$ are regular on $\Delta$, we have $b^\Delta=d^\Delta$ and set $x=d$. Now assume $m=4$. It is easy to check that if a product of two elements of order $4$ in $\Sy_4$ has $2$-power order, then the two elements are equal or inverse to each other. Since $bd\in P$ has $2$-power order, we conclude that $b^\Delta=d^\Delta$ or $b^\Delta=(d^{-1})^\Delta$. Set $x=d$ in the former case, and $x=d^{-1}$ in the latter case. Then summarizing this paragraph, we obtain $b^\Delta=x^\Delta$ with $x=d^{\pm1}$. Consequently, $bx^{-1}$ fixes every element in $\Delta$.

Since $h^d=h^r$ for every $h\in H$ and $r=\pm1$, we have $h^{d^{-1}}=h^r$ for every $h\in H$. This together with $x=d^{\pm1}$ gives $h^x=h^r=h^b$ for every $h\in H$, which indicates that $bx^{-1}$ centralizes $H$. For each $i\in\{1,\dots,m\}$, since $bx^{-1}$ fixes $\delta_i$ and $\Omega_i$ is the orbit of $H$ containing $\delta_i$, it follows that $bx^{-1}$ fixes every element in $\Omega_i$. Hence $bx^{-1}=1$, and so $\langle b\rangle=\langle x\rangle=\langle d\rangle$. As $X^{\alpha\beta}=H\rtimes\langle c^{\alpha\beta}\rangle=H\rtimes\langle d\rangle$, this shows that $X^{\alpha\beta}=H\rtimes\langle b\rangle=G$, which completes the proof.
\end{proof}
%%%%%%%%%%%%%%%%%%%%%%%%%%%%%%%%%%%%%%%%%%%%%%%%%%%%%%%%%%%%%%%%%%%%

Based on Lemma~\ref{sub-conjugate}, we may prove the following:

\begin{lem}\label{p-CI-cyclic}
Let $G$ be a cyclic group of order $2^{\ell}n$ with $\ell\in \{0,1,2\}$ and $n$ odd, and let $p$ be the least prime divisor of $n$. Then every connected Cayley digraph of $G$ with valency at most $p$ is a CI-digraph.
\end{lem}

\begin{proof}
Write $G=\langle a\rangle\cong \mathbb{Z}_{2^{\ell}n}$. Let $\Gamma=\Cay(G,S)$ be a connected Cayley digraph with $|S|\leq p$, and let $A=\Aut(\Gamma)$. If $\ell=0$, then since $G$ is a connected $p$-DCI-group (~\cite[Theorem~1.1]{C.H.Li2}), $\Gamma$ is a CI-digraph, as required. Next we consider the case $\ell\in \{1,2\}$. Denote by $A_1$ the stabilizer of $1$ in $A$.

Assume that $p$ does not divide $|A_1|$. Since $\Gamma$ is connected and has valency at most $p$, each prime divisor of $|A_1|$ is at most $p$. Then as $p$ is the least prime divisor of $n$, we conclude that $|A_1|$ is coprime to $n$. Let $\pi$ be the set of prime divisor of $n$. It follows from $A=R(G)A_1$ that $\langle a^{2^\ell}\rangle$ is a Hall $\pi$-subgroup of $A$. By \cite[Theorem 9.1.10]{Robinson}, all nilpotent Hall $\pi$-subgroup of $A$ are conjugate. Hence all subgroups isomorphic to $\langle a^{2^\ell}\rangle$ are conjugate in $A$, and so all regular subgroups of $A$ isomorphic to $R(G)$ are conjugate by Lemma~\ref{sub-conjugate}. This shows that $\Gamma$ is a CI-digraph by Proposition~\ref{CI-graph-prop}.

Assume that $p$ divides $|A_1|$. If $\Gamma$ has valency less than $p$, then the connectivity of $\Gamma$ means that $|A_1|$ is not divisible by $p$, a contradiction. Thus $\Gamma$ has valency $p$, and it further follows from $p$ dividing $|A_1|$ that $\Gamma$ is arc-transitive. Then by \cite[Theorem 1.3]{C.H.Li9}, every connected arc-transitive Cayley digraph over a cyclic group is a CI-digraph, and hence $\Gamma$ is a CI-digraph. This completes the proof.
\end{proof}
%%%%%%%%%%%%%%%%%%%%%%%%%%%%%%%%%%%%%%%%%%%%%%%%%%%%%%%%%%%%%%%%%%%%

Let $X$ and $Y$ be digraphs. The {\em lexicoproduct $X[Y]$} of $X$ and $Y$ is defined as the digraph with vertex set $V(X)\times V(Y)$ such that $((x_1,y_1),(x_2,y_2))$, where $x_1,x_2\in V(X)$ and $y_1,y_2\in V(Y)$, is an arc if and only if $(x_1,x_2)\in\Arc(X)$, or $x_1=x_2$ and $(y_1,y_2)\in \Arc(Y)$. We now give some sufficient conditions for CI-digraphs of generalized quaternion groups $\mathrm{Q}_{4n}$ with $n\geq 3$ odd.

\begin{lem}\label{p-CI-Q}
Let $\Gamma=\Cay(\mathrm{Q}_{4n},S)$ be a connected Cayley digraph of $\mathrm{Q}_{4n}$ with $n\geq 3$ odd, and let $A=\Aut(\Gamma)$. Then the following statements hold:
\begin{enumerate}[{\rm (a)}]
\item If $|A_1|$ is coprime to $n$, then $\Gamma$ is a CI-digraph.
\item If $n$ is a power of a prime $p$ and $\Gamma$ is arc-transitive with $|S|=p$, then  $\Gamma$ is a CI-digraph.
\end{enumerate}
\end{lem}

\begin{proof}
Let $G=\mathrm{Q}_{4n}=\langle a,b\mid a^{2n}=1,\,b^2=a^n,\,a^b=a^{-1}\rangle$, and let $\pi$ be the set of prime divisors of $n$.

To prove part~(a), suppose that $|A_1|$ is coprime to $n$. Since $A=R(G)A_1$ and $n$ is odd, we conclude that $\langle R(a^2)\rangle$ is a Hall $\pi$-subgroup of $A$. Since all nilpotent Hall $\pi$-subgroups of $A$ are conjugate by~\cite[Theorem 9.1.10]{Robinson}, Lemma~\ref{sub-conjugate} implies that all regular subgroups of $A$ isomorphic to $R(G)$ are conjugate. Hence $\Gamma$ is a CI-digraph by Proposition~\ref{CI-graph-prop}. This proves part~(a).

To prove part (b), suppose that $n=p^\ell$ for an odd prime $p$ and a positive integer $\ell$, and that  $\Gamma$ is arc-transitive with $|S|=p$. If $\ell=1$, then $\mathrm{Q}_{4p}$ is a DCI-group by \cite{Mu4}, and so $\Gamma$ is a CI-digraph.

From now on we assume that $\ell\geq 2$. Since $A_1$ acts transitively on $S$, the order $|A_1|$ is divisible by $p$, and so $|A|=|R(G)||A_1|=4n|A_1|$ is divisible by $p^{\ell+1}$. Write
\[
H=\langle a^2\rangle\ \text{ and }\ N=\mathrm{N}_A(R(H)).
\]
Since $|R(H)|=|H|=p^{\ell}$ and $|A|$ is divisible by $p^{\ell+1}$, it is clear that $R(H)$ is not a Sylow $p$-subgroup of $A$. By Sylow Theorem and Proposition~\ref{proper-p-group}, $|N|$ is divisible by $p^{\ell+1}$. Since $R(H)\unlhd R(G)$, we get $R(G)\leq N$. It follows that $\Gamma$ is $N$-vertex-transitive, and
\begin{equation}\label{Eqn1}
|N_u|=|N|/|R(G)| \text{ is divisible by } p \text{ for every } u\in V(\Gamma).
\end{equation}
Hence $\Gamma$ is $N$-arc-transitive as $|S|=p$. Since $R(H)$ is characteristic in $R(G)$ and $R(G)\unlhd N$, we have $R(H)\unlhd N$.
Since $|S|=p$, it follows that $\Gamma$ is $N$-locally primitive. The orbit set of $R(H)$ on $V(\Gamma)$ is $\{H,bH,b^2H,b^3H\}=V(\Gamma_{R(H)})$. By Lemma~\ref{no-arc}\,(b), either $R(H)$ is the kernel of $N$ on $V(\Gamma_{R(H)})$ and $\Gamma$ is a $R(H)$-cover of $\Gamma_{R(H)}$, or $\Gamma_{R(H)}$ is the directed cycle $\overrightarrow{\mathrm{C}}_4$ of length $4$.

Assume that $R(H)$ is the kernel of $N$ on $V(\Gamma_{R(H)})$ and $\Gamma$ is a $R(H)$-cover of $\Gamma_{R(H)}$. Then $\Gamma_{R(H)}$ has order $4$ and out-valency $p\geq3$. Hence $p=3$. According to~\cite[Theorem 1.4]{Ma}, $\mathrm{Q}_{4p^\ell}$ is a 3-DCI-group. Hence $\Gamma$ is a CI-digraph, as required.

Now assume that $\Gamma_{R(H)}=\overrightarrow{\mathrm{C}}_4$. Since $\Gamma$ is connected, we have $ba^i\in S$ for some integer $i$. Note that there is an automorphism $\alpha$ of $G$ sending $a$ and $b$ to $a$ and $ba^i$, respectively. Then replacing $S$ by $S^\alpha$, we may assume that $b\in S$, whence
\begin{equation}\label{Eqn2}
\Arc(\Gamma_{R(H)})=\{(H,bH),(bH,b^2H),(b^2H,b^3H),(b^3H,H)\}.
\end{equation}
Let $C=\mathrm{C}_A(R(H))$ and let $K$ be the kernel of $C$ acting on $V(\Gamma_{R(H)})$. Then $R(H)\leq C$, $R(b^2)\in C$, $C\unlhd N$, and $C/K\leq\Aut(\Gamma_{R(H)})=\Aut(\overrightarrow{\mathrm{C}}_4)\cong\mathbb{Z}_4$. Note that $b^iH$ is an orbit for both $R(H)$ and $K$. By the Frattini argument, $K=R(H)K_u$ for $u\in V(\Gamma)$. As $\Gamma_{R(H)}=\overrightarrow{\mathrm{C}}_4$, it follows that $C_u$ fixes $V(\Gamma_{R(H)})$ pointwise. Hence $C_u\leq K$ and $C_u=K_u$. Since $K\leq C=\mathrm{C}_A(R(H))$, we obtain
\begin{equation}\label{Eqn3}
K=R(H)\times C_u\text{ for every }u\in V(\Gamma).
\end{equation}
Consequently, $C_1C_{b^2}\unlhd K$. Noting that $R(H)$ is a $p$-group, it follows from~\eqref{Eqn3} that $|K|_{p'}=|C_1|_{p'}=|C_{b^2}|_{p'}=|C_1C_{b^2}|_{p'}$. Since $|C_1\cap C_{b^2}|=|C_1||C_{b^2}|/|C_1C_{b^2}|$,
this implies
\[
|C_1\cap C_{b^2}|_{p'}=|K|_{p'}.
\]

Suppose for a contradiction that $K=R(H)$. Then \eqref{Eqn3} implies that $C_1=1$, and so $N_1$ acts faithfully on $R(H)$ by conjugation. Hence $N_1\leq \Aut(R(H))\cong\mathbb{Z}_{p^{\ell-1}(p-1)}$ is cyclic. This together with \eqref{Eqn1} implies that $N_1$ has a unique subgroup of order $p$, say $P$. Let $L$ be the kernel of $N$ acting on $V(\Gamma_{R(H)})$. Since $\Gamma_{R(H)}=\overrightarrow{\mathrm{C}}_4$, it follows that $N_1$ fixes $V(\Gamma_{R(H)})$ pointwise, which means that $N_1=L_1$. Thus, by the Frattini argument, $L=R(H)N_1$. Consequently, $L/R(H)$ is cyclic. Write $M=R(H)P$. Then $M/R(H)$ is the unique subgroup of order $p$ of $L/R(H)$ and so characteristic in $L/R(H)$. Note that $L/R(H)\unlhd N/R(H)$. Then $M/R(H)\unlhd N/R(H)$. This implies that $R(H)\leq M\unlhd N$, and so all orbits of $M$ on $V(\Gamma)$ have length $|R(H)|$. Clearly,
\[
R(H)P=M=R(H)M_1=R(H)M_b.
\]
Since $|M|=|R(H)||P|=p|R(H)|$, we obtain $|M_1|=p=|M_b|$. Hence both $M_1$ and $M_b$ are cyclic groups of order $p$. Recall that $\Aut(R(H))\cong\mathbb{Z}_{p^{\ell-1}(p-1)}$ and $\ell\geq 2$. The unique subgroup of order $p$ of $\Aut(R(H))$ is generated by the automorphism $\gamma$ of $R(H)=\langle R(a^2)\rangle$ defined by
\[
\gamma: R(a^2)\mapsto R(a^2)^r=R(a^{2r}),\ \text{ where } r:=p^{\ell-1}+1.
\]
Since the action of $M_1\leq N_1$ by conjugation on $R(H)$ is faithful, it follows that
\[
R(a^2)^{\alpha}=R(a^2)^{\gamma}=R(a^{2r}) \text{ for some generator } \alpha \text{ of } M_1.
\]
For integers $i$ and $j$, since $a^2$ has order $n=p^\ell$ and $r^j\equiv jp^{\ell-1}+1\pmod{p^\ell}$, we have
\begin{equation}\label{Eqn4}
R(a^{2i})^{\alpha^j}=R(a^{2ir^j})=R(a^{2i(jp^{\ell-1}+1)})\in R(a^{2i})\langle R(a^{2p^{\ell-1}})\rangle.
\end{equation}
Take arbitrary $x,y\in M$. Since $M=R(H)M_1$, we may write $x=x_1x_2$ and $y=y_1y_2$ with $x_1,y_1\in R(H)$ and $x_2,y_2\in M_1$. Then the commutator
\[
[x,y]=[x_1x_2,y_1y_2]=(x_1x_2)^{-1}(y_1y_2)^{-1}(x_1x_2)(y_1y_2)=(x_1^{-1})^{x_2}(y_1^{-1}x_1)^{x_2y_2}(y_1)^{y_2}.
\]
This together with \eqref{Eqn4} implies that $[x,y]\in \langle R(a^{2p^{\ell-1}})\rangle$. Hence the derived group
\[
M'=\langle R(a^{2p^{\ell-1}})\rangle\cong\mathbb{Z}_p.
\]
Since $M=M_bR(H)$, we may write $\alpha=\beta R(a^2)^t$ for some $\beta$ of $M_b$ and integer $t$. Since $|M'|=p$, we derive from Proposition~\ref{regular-p-group} that $(R(a^2)^t)^p=(\beta^{-1}\alpha)^p=(\beta^{-1})^p\alpha^p=1$. Therefore, $t$ is divisible by $p^{\ell-1}$.
In particular, $t$ is divisible by $p$ as $\ell\geq 2$. Since
\[
b^\alpha=b^{\beta R(a^2)^t}=b^{R(a^2)^t}=ba^{2t},
\]
we derive for each integer $k$ that
\[
(ba^{2tk})^\alpha=b^{R(a^{2tk})\alpha}=b^{\alpha R(a^{2tk})^\alpha}=b^{\alpha R(a^{2tkr})}=(ba^{2t})^{R(a^{2tkr})}=ba^{2t(1+kr)}.
\]
Hence $\alpha$ stabilizes $b\langle a^{2t}\rangle$, and so $M_1=\langle \alpha\rangle$ stabilizes $b\langle a^{2t}\rangle$. Note that the stabilizer $M_1$ is transitive or trivial on the out-neighborhood $\Gamma^+(1)=S$ of $1$ in $V(\Gamma)$. If $M_1$ is trivial on $S$, then we obtain a contradiction that $M_1=1$ as $\Gamma$ is $N$-vertex-transitive and strongly connected. Hence $M_1$ is transitive on $S$, and so $S=b^{M_1}$ as $b\in S$. Then $S=b^{M_1}\subseteq b\langle a^{2t}\rangle$, and so $\langle S\rangle\leq \langle b,a^p\rangle<G$ as $p$ divides $t$. This contradicts the connectivity of $\Gamma$.

Thus we have $K\neq R(H)$.
Suppose for a contradiction that $N$ has an imprimitive block system on $V(\Gamma)$ such that each block is an independent set of size $p$ and the induced subdigraph of each two blocks is either $\overline{\mathrm{K}}_{2p}$ or $\overrightarrow{\mathrm{K}}_{p,p}$. Let $\Delta$ be an imprimitive block containing $1$. Then since $R(G)$ is a regular subgroup of $N$, we derive that $\Delta$ is a subgroup of $G$ and $S$ is a union of left cosets of $\Delta$ in $G$. Since $b\in S$ and $|S|=p$, it follows that $\Delta=\langle a^{2p^{\ell-1}}\rangle$ and $S=b\Delta=b\langle a^{2p^{\ell-1}}\rangle$. This contradicts the condition $\langle S\rangle=G$ by the connectivity of $\Gamma$. Thus $N$ does not have such an imprimitive block system. However, we prove in the following, distinguishing two cases, that $\Gamma\cong\overrightarrow{\mathrm{C}}_{4p^{\ell-1}}[\overline{\mathrm{K}}_p]$, which will complete the proof of part~(b).

\smallskip
\noindent{\bf Case 1:} $C_1\cap C_{b^2}=1$.

\smallskip
Recall that $R(H)\times C_1=K\neq R(H)$. Then $C_1\neq 1$, and $|C_1|_{p'}=|K|_{p'}=|C_1\cap C_{b^2}|_{p'}=1$. This means that $C_1$ is a $p$-group. Since $C/K\leq\mathbb{Z}_4$, it follows that $K=R(H)\times C_1$ is a Sylow $p$-subgroup of $C$ and thus characteristic in $C$. Note that
\[
C_1\cong C_1/(C_1\cap C_{b^2})\cong C_1C_{b^2}/C_{b^2}\leq K/C_{b^2}\cong R(H)
\]
is cyclic. Hence $K$ has a characteristic subgroup $D=X\times Y\cong \mathbb{Z}_p^2$, where $\mathbb{Z}_p\cong X\leq R(H)$ and $\mathbb{Z}_p\cong Y\leq C_1$. Then $D$ is characteristic in $C$. As $C\unlhd N$, we have $D\unlhd N$. Since $X$ is semiregular of order $p$ and $Y$ fixes the vertex $1$, we then conclude that the orbits of $D=YX$ on $V(\Gamma)$ all have length $p$. For every $u\in V(\Gamma)$, it follows that $D=XD_u$, and so $D_u\cong\mathbb{Z}_p$ is either transitive or trivial on the out-neighborhood $\Gamma^+(u)$ of $u$. If $D_u$ is trivial on $\Gamma^+(u)$, then $D_u=1$ as $\Gamma$ is $N$-vertex-transitive and strongly connected, contradicting to $D_u\cong\mathbb{Z}_p$. Thus $D_u$ is transitive on $\Gamma^+(u)$ for every $u\in V(\Gamma)$. This implies that if $\Delta_1$ and $\Delta_2$ are two orbits of $D$ and there is an arc from some vertex of $\Delta_1$ to some vertex of $\Delta_2$, then $(x,y)\in \Arc(\Gamma)$ for all $x\in \Delta_1$ and $y\in \Delta_2$. Since $\Gamma$ has out-valency $p$, it follows that $\Gamma\cong \overrightarrow{\mathrm{C}}_{4p^{\ell-1}}[\overline{\mathrm{K}}_p]$, as required.

\smallskip
\noindent{\bf Case 2:} $C_1\cap C_{b^2}\neq1$.

\smallskip
Recall that $\Gamma_{R(H)}=\overrightarrow{\mathrm{C}}_4$ and $b\in S$, we have $S\cap(H\cup b^2H\cup b^3H)=\emptyset$. From $C/K\leq\mathbb{Z}_4$ we deduce that $B:=K\langle R(b^2)\rangle\unlhd C$. Since $R(b^2)$ interchanges $C_1$ and $C_{b^2}$ by conjugation, we have $C_1\cap C_{b^2}\unlhd B$. Note that $H\cup b^2H$ and $bH\cup b^3H$ are the orbits of $B$ on $V(\Gamma)$. Then the orbits of $C_1\cap C_{b^2}$ on $bH\cup b^3H$ have the same length, say $t$. Hence the valency of $\Gamma$ is a multiple of $t$. As $\Gamma$ is $p$-valent, we deduce that $t=1$ or $p$. Recall that
\[
K=R(H)\times C_1=R(H)\times C_{b^2}.
\]
The group $C_1\cap C_{b^2}$ fixes $H\cup b^2H$ pointwise. If $t=1$, then $C_1\cap C_{b^2}$ fixes both $H\cup b^2H$ and $bH\cup b^3H$ pointwise, which means that $C_1\cap C_{b^2}=1$, a contradiction. Thus $t=p$, that is, the orbits of $C_1\cap C_{b^2}$ on $bH\cup b^3H$ all have length $p$. Since $R(b)\in N$ normalizes $C$, it follows that $C_b\cap C_{b^3}=(C_1\cap C_{b^2})^{R(b)}$ fixes $(H\cup b^2H)^{R(b)}=bH\cup b^3H$ pointwise and that the orbits of $C_b\cap C_{b^3}$ on $(bH\cup b^3H)^{R(b)}=H\cup b^2H$ all have length $p$. Let
\[
T=(C_1\cap C_{b^2})(C_b\cap C_{b^3}).
\]
Then all orbits of $T$ on $V(\Gamma)$ have length $p$. Note that $C_1\cap C_{b^2}\leq T_v$ for every $v\in H\cup b^2H$ and $C_b\cap C_{b^3}\leq T_w$ for every $w\in bH\cup b^3H$. Then we derive from~\eqref{Eqn2} that the stabilizer $T_u$ is transitive on the out-neighbors of $u$ in $\Gamma$ for every $u\in V(\Gamma)$. This implies that if $\Delta_1$ and $\Delta_2$ are two orbits of $T$ and there exists an arc from some vertex of $\Delta_1$ to some vertex of $\Delta_2$, then $(x,y)\in \Arc(\Gamma)$ for all $x\in \Delta_1$ and $y\in \Delta_2$. Hence $\Gamma\cong \overrightarrow{\mathrm{C}}_{4p^{\ell-1}}[\overline{\mathrm{K}}_p]$, as required.
\end{proof}
%%%%%%%%%%%%%%%%%%%%%%%%%%%%%%%%%%%%%%%%%%%%%%%%%%%%%%%%%%%%%%%%%%%%

Let $\Gamma$ be a connected Cayley digraph of a finite group $G$ of valency $m<p$, and let $A=\Aut(\Gamma)$. By the same argument as~\cite[Lemma 2.1]{C.H.Li0} we see that every prime divisor of $|A_1|$ is less than $p$. Thus the following result is a consequence of Lemma~\ref{p-CI-Q}.

\begin{lem}\label{p-power-CI-Q}
Let $n$ be a power of an odd prime $p$, let $\Gamma=\Cay(\mathrm{Q}_{4n},S)$ be a connected Cayley digraph of $\mathrm{Q}_{4n}$ with $|S|\leq p$. Then $\Gamma$ is a CI-digraph.
\end{lem}
%%%%%%%%%%%%%%%%%%%%%%%%%%%%%%%%%%%%%%%%%%%%%%%%%%%%%%%%%%%%%%%%%%%%

Now we are ready to prove Theorem~\ref{mainth2}.

\begin{proof}[Proof of Theorem~\ref{mainth2}.]
Let $n=p^\ell$, where $p$ is a prime and $\ell$ is a positive integer, let $G=\mathrm{Q}_{4n}=\langle a,b\mid a^{2n}=1,\, b^2=a^n,\,a^b=a^{-1}\rangle$, and let $m$ be an integer with $1\leq m\leq 2n-1$.

First, we suppose that $G$ has the $m$-DCI property. By Theorem~\ref{mainth1}, $n$ is odd, and so $p$ is odd. If $\ell \geq 2$, then it follows from Theorem~\ref{mainth1} that $m \leq p$. This shows that either $n=p$ or $m\leq p$, which completes the proof of the necessity.

Next, we prove the sufficiency. So suppose that $p$ is odd and either $n=p$ or $m\leq p$. If $n=p$, then since $\mathrm{Q}_{4p}$ is a DCI-group (by~\cite{Mu4}), $G$ has the $m$-DCI property. Now assume $m\leq p$. Let $\Cay(G,S)$ be a Cayley digraph with $|S|=m$, and let $\Cay(G,T)$ be a Cayley digraph isomorphic to $\Cay(G,S)$. Since $\Cay(G,S)\cong \Cay(G,T)$, we have $\Cay(\langle S\rangle,S)\cong \Cay(\langle T\rangle,T)$, which implies that $|\langle S\rangle|=|\langle T\rangle|$. As $G$ is a generalized quaternion group of order $4p^\ell$ with $p$ odd prime, it follows that $\langle S\rangle\cong\langle T\rangle$. According to Lemma~\ref{homogeneous}, there exists $\delta\in\Aut(G)$ such that $\langle T\rangle^\delta=\langle S\rangle$. Then we have
\[
\Cay(\langle T\rangle,T)\cong \Cay(\langle T\rangle^\delta,T^\delta)=\Cay(\langle S\rangle,T^\delta),
\]
and hence $\Cay(\langle S\rangle,S)\cong \Cay(\langle S\rangle,T^\delta)$. Set $\Gamma=\Cay(\langle S\rangle,S)$. Then $\Gamma$ is a connected $m$-valent Cayley digraph with $m=|S|\leq p$. As a subgroup of $\mathrm{Q}_{4n}$, we see that $\langle S\rangle$ is either a cyclic or a generalized quaternion subgroup of $\mathrm{Q}_{4n}$. Since Lemmas~\ref{p-CI-cyclic} and~\ref{p-power-CI-Q} assert that $\Gamma$ is a CI-digraph, there is an automorphism of $\langle S\rangle$ mapping $S$ to $T^\delta$. Again by Lemma~\ref{homogeneous}, this automorphism can be extended to an automorphism
of $G$, say $\gamma$. Then $S^\gamma=T^{\delta}$, and by taking $\sigma=\gamma\delta^{-1}$ we have $\sigma\in \Aut(G)$ and $S^\sigma=T$. This shows that $G$ has the $m$-DCI property, proving the sufficiency.
\end{proof}

%%%%%%%%%%%%%%%%%%%%%%%%%%%%%%%%%%%%%%%%%%%%%%%%%%%%%%%%%%%%%%%%%%%%
\medskip
\noindent {\bf Acknowledgements:}
We would like to thank an anonymous referee for careful reading of the paper and helpful suggestions. The work was supported by the National Natural Science Foundation of China (12331013, 12311530692, 12271024, 12161141005), the 111 Project of China (B16002) and the scholarship No.~202207090064 from the China Scholarship Council. The work was done during a visit of the first author to The University of Melbourne. The first author would like to thank The University of Melbourne for its hospitality and Beijing Jiaotong University for consistent support.

%%%%%%%%%%%%%%%%%%%%%%%%%%%%%%%%%%%%%%%%%%%%%%%%%%%%%%%%%%%%%%%%%%%%

\end{document}